\documentclass[a4paper]{amsart}

\usepackage[utf8]{inputenc}
\usepackage{amsmath, amssymb, amsthm, xypic, csquotes, mathtools, tikz-cd, verbatim}

\topmargin = - 1.5 cm
\newtheorem{theorem}{Theorem}[section]
\newtheorem{prop}[theorem]{Proposition}

\newtheorem{lemma}[theorem]{Lemma}
\newtheorem{corollary}[theorem]{Corollary}

\theoremstyle{definition}
\newtheorem{construction}[theorem]{Construction}
\newtheorem{remark}[theorem]{Remark}
\newtheorem*{acknowledgement}{Acknowledgement}

\newcommand{\Nrd}{\mathrm{Nrd}}
\newcommand{\Z}{\mathbb{Z}}

\newcommand{\im}{\mathop{\mathrm{im}}}
\newcommand{\Aut}{\mathrm{Aut}}
\newcommand{\Gal}{\mathrm{Gal}}
\newcommand{\Mat}{\mathrm{Mat}}

\newcommand{\tr}{\mathrm{tr}}

\newcommand{\x}{\mathbf{x}}
\newcommand{\y}{\mathbf{y}}

\title{Finite subgroups of automorphism groups of Severi--Brauer varieties}
\author{Anna Savelyeva}
\date{}

\begin{document}

\begin{abstract}
We discuss the structure of finite subgroups acting on minimal Severi--Brauer varieties and provide a complete description of such groups for the varieties of dimension $p-1$, where $p \geqslant 3$ is prime. \\[-1 cm]
\end{abstract}

\maketitle

\section{Introduction}

An algebraic variety $X$ of dimension $n$ over a field $k$ is called a \textit{Severi--Brauer variety} if it becomes isomorphic to $\mathbb{P}^n_{\bar k}$ after the extension of scalars to the algebraic closure $\bar k$ of $k$. We call a Severi--Brauer variety \textit{non-trivial}
 if it is not isomorphic to $\mathbb{P}^n_k$ over the base field~$k$.

Given a field $k$ and a positive integer $n$, there exists a one-to-one correspondence between Severi--Brauer varieties of dimension $n-1$ and central simple algebras of degree $n$ which, moreover, preserves automorphism groups (see~{\cite[Ch. III, \S 
\,III]{Chat} or \cite[Ch. 5]{GS}}). The correspondence is important in the study of Severi--Brauer varieties and their automorphisms, as it allows to study the geometric object using purely algebraic methods. According to Skolem--Noether's theorem (see~\cite[Thr. 2.7.2]{GS}) all automorphisms of central simple algebras are inner, so, given a central simple algebra $A$ over a field~$k$, its automorphism group $\Aut(A)$ is isomorphic to~$A^*/k^*$. Therefore, as mentioned above, the automorphism group of the corresponding Severi--Brauer variety will be isomorphic to $A^*/k^*$ as well.

Nevertheless, the structure of automorphism groups of non-trivial Severi--Brauer varieties is still complicated. One can then ask whether we could describe the finite subgroups of these groups. However, in this generality the question does not make much sense. In fact, given a division algebra~$D$, the matrix algebras $\Mat_n(D)$ are central simple for all positive~$n$, and for every finite group $G$ there exists a number~$n$ large enough such that $G$ acts faithfully on~$\Mat_n(D)$. 

Wedderburn theorem (see {\cite[Thr. 2.1.3]{GS}}) states that, in fact,
every finite dimensional central simple algebra over a field $k$ is isomorphic to a matrix algebra $\Mat_n(D)$ with coefficients in a central division algebra $D$ over~$k$. Therefore, the following variation of the general question is natural: what are the finite subgroups of automorphism groups of Severi--Brauer varieties that correspond to division algebras? Such Severi--Brauer varieties are called \textit{minimal} and have a geometric meaning: those are exactly the Severi--Brauer varieties that have no non-trivial twisted linear Severi--Brauer subvarieties (see, for example, \cite[Def. 26]{Kollar}). \\[-5pt]

The goal of this text is to study finite subgroups of automorphism groups of division algebras and, accordingly, of minimal Severi--Brauer varieties. We prove the following theorem.

\begin{theorem} \label{main}
Let $A$ be a division algebra of degree $n$ over a field $k$ of characteristic coprime with~$n$. Then for every finite subgroup $$G \subset \Aut(A) \cong A^*/k^*$$ there exists a finite subgroup $\widetilde{N}_G \subset A^*$ such that its image $N_G \subset A^*/k^*$ under the canonical projection is a normal subgroup of $G$ and the quotient $G/N_G$ is an abelian group of order dividing~$n^2$.
\end{theorem}

Using the classification of finite subgroups of multiplicative groups of division algebras due to Amitsur (see \cite{Amitsur}), one can list all possible isomorphism classes of $\widetilde{N}_G$ and thus of $N_G$. In particular, one can deduce the following.

\begin{corollary} \label{A5}
The only finite simple non-abelian group that can faithfully act on a minimal Severi--Brauer variety is the icosahedral group $\mathfrak{A}_5$.
\end{corollary}

\begin{remark} \label{realConic}
Note that the group $\mathfrak{A}_5$ acts faithfully, for example, on the real conic given by the equation $x^2 + y^2 + z^2 = 0$ in $\mathbb{P}^2_\mathbb{R}$.
\end{remark}

Due to Wedderburn theorem every non-trivial Severi--Brauer variety of dimension $p-1$ is minimal when $p$ is prime. In the case when $p = 3$ and the characteristic of the base field is zero, all finite subgroups of automorphism groups of non-trivial Severi--Brauer surfaces were completely classified by Shramov in \cite{surfaces}. 

Following \cite{surfaces}, we call a semidirect product $H = \Z/n\Z \rtimes \Z/m\Z$ of two cyclic groups \textit{balanced} if the group $H$ has a trivial center. Note that, given a prime $p$ and a positive integer $n$, a balanced semidirect product exists if and only if every prime $q$ that divides $n$ is congruent to $1$ modulo $p$. A balanced semidirect product is unique when it exists. In \cite{surfaces} the finite groups acting on Severi--Brauer surfaces were described as balanced semidirect products of particular cyclic groups. As a corollary of Theorem \ref{main} we prove the following generalization of this result.

\begin{theorem} \label{mainPrime}
Let $G$ be a finite subgroup of an automorphism group of a non-trivial Severi--Brauer variety of dimension $p - 1$ over a field $k$ of characteristic coprime with $p$, where $p \geqslant 3$ is prime. Then there exists a positive integer $n$ such that  $G$ is isomorphic to a subgroup of $(\Z/n\Z \rtimes \Z/p\Z) \times \Z/p\Z$, where the semidirect product is balanced.
\end{theorem}

\begin{remark}
As one can see from Remark~\ref{realConic}, Theorem~\ref{mainPrime} does not hold for $p = 2$.
\end{remark}

The structure of the text is as follows. In Section 2 we introduce the construction of the group $N_G$ mentioned in Theorem~\ref{main}, and study its basic properties. Section 3 is devoted to the discussion of several results on the structure of $p$-subgroups of automorphism groups of central division algebras. In Section~4 we present several field theory results to be used later. Section~5 contains the proof of Theorem~\ref{main} in the case when the group $N_G$, constructed in Section~2, is trivial. The proof of Theorem~\ref{main} is completed in Section~6. In Section~7 we concentrate on division algebras of prime degree and prove Theorem~\ref{mainPrime}.

\begin{acknowledgement}
I am very grateful to Constantin Shramov for stating the problem, numerous valuable discussions, as well as great patience throughout the whole process of the text creation. I am also grateful to Andrey Trepalin and Sergey Tikhonov for the detailed revision of the text.
\end{acknowledgement}

\section{Lifting construction} \label{N_Gsect}
In this section, given a finite subgroup $G \subset A^*/k^*$ for a central simple algebra $A$, we try to construct its subgroup that can be lifted to a finite subgroup of $A^*$ under the canonical projection and that is as large as possible. It is tempting to consider all the elements in $G$ that have at least one preimage of finite order in~$A^*$, however it is a priori unclear neither that they form a subgroup in $G$, nor that this subgroup itself admits a finite lift. Surprisingly, both of the assertions turn out to be true. To prove them we use the existence of the reduced norm map $\Nrd_A \colon A^* \to k^*$, and the fact  that its kernel restricted to the base field $k^* \subset A^*$ is finite. We refer to \cite[Sect. 2.6]{GS} for more details. We denote by~$f_A \colon A^* \to A^*/k^*$ the canonical projection. \\[-0.7 cm]

$$\xymatrix{
A^* \ar[r]^{\Nrd} \ar[d]^{f_A} & k^* \\
A^*/k^*
}$$

We start with the following proposition.

\begin{prop} \label{finite_order}
Let $A$ be a central simple algebra over a field $k$ and let $G$ be a finite subgroup of $A^*/k^*$. Then an element $g \in f_A^{-1}(G) \subset A^*$ is of finite order if and only if its reduced norm is a root of unity.
\end{prop}

\begin{proof}
Clearly, the reduced norm of an element of finite order is a root of unity. To prove the inverse, note that as $f_A(g) \in G$ is an element a finite order and $\ker f_A = k^*$, there exists a positive integer $\alpha$ such that $g^\alpha$ lies in $k^*$. As $g$ has a reduced norm of finite order, so does $g^\alpha$, hence there exists a positive integer $\beta$ such that $(g^\alpha)^\beta$ is an element of $k^*$ of reduced norm $1$. However, as mentioned above, the reduced norm map restricted to the base field $k^* \subset A^*$ has finite kernel, thus the element $g^{\alpha \beta}$ is of finite order in $A^*$.
\end{proof}

\begin{corollary} \label{N_G}
Let $A$ be a central simple algebra over a field $k$ and let $G$ be a finite subgroup of~$A^*/k^*$. Then all the elements of $G$ that admit a lift of finite order in $A^*$ form a normal subgroup $N_G \lhd G$, such that $G/N_G$ is abelian.
\end{corollary}

\begin{proof}
Let $H \subset k^*$ be a group of all roots of unity in $k^*$. Denote by $N_G \subset G$ the group $f_A(\Nrd_A^{-1}(H))\cap G$ of all elements that admit a preimage whose reduced norm is in $H$. By Proposition~\ref{finite_order} those are exactly the elements that admit a lift of finite order in $A^*$. This property is clearly preserved under conjugation, thus $N_G$ is a normal subgroup of $G$. It remains to prove that $G/N_G$ is abelian.

Note that $\Nrd_A$ defines an isomorphism of $f_A^{-1}(G)/(f_A^{-1}(G) \cap \Nrd_A^{-1}(H))$ with a subgroup of~$k^*/H$, therefore $f_A^{-1}(G)/(f_A^{-1}(G) \cap \Nrd_A^{-1}(H))$ is abelian. As $f_A^{-1}(G)$ contains the kernel of $f_A$, the image under $f_A$ of $f_A^{-1}(G) \cap \Nrd_A^{-1}(H)$ is exactly $N_G$. Thus, $f_A$ induces a surjection $$f_A^{-1}(G)/(f_A^{-1}(G) \cap \Nrd_A^{-1}(H)) \to G/N_G,$$ which completes the proof.
\end{proof}

We now want to prove that there exists a finite subgroup $\widetilde{N}_G \subset A^*$ such that \mbox{$f_A(\widetilde{N}_G) = N_G$.} Note that the group $\Nrd_A^{-1}(H)\cap f_A^{-1}(G)$ that consists of all the elements of finite order in the preimage of~$G$, which is an obvious candidate for $\widetilde{N}_G$, is not always finite, as the field $k$ might contain an infinite number of roots of unity. In what follows we show that in order to lift $N_G$ it is enough to consider a finite subgroup $H'$ of $H$, thus the finite lift $\widetilde{N}_G$ will be defined as~$\Nrd_A^{-1}(H')\cap f_A^{-1}(G)$.

\begin{prop} \label{N_GLift}
Let $A$ be a central simple algebra over a field $k$. Let $G$ be a finite subgroup of~$A^*/k^*$ and let $N_G$ be the subgroup of $G$ constructed as in Corollary~\ref{N_G}. Then there exists a finite group $\widetilde{N}_G \subset A^*$ such that $f_A(\widetilde{N}_G) = N_G$.
\end{prop}

\begin{proof}
For every element $x_i \in N_G$ choose one of its preimages $g_i \in f_A^{-1}(x_i)$ of finite order. Denote by $h_i$ the reduced norm of $g_i$. As every $h_i$ is a root of unity, the subgroup $H'$ of $k^*$ generated by all the $h_i$ is finite. Moreover, by construction $N_G$ lies in $f_A(\Nrd^{-1}_A(H'))$, thus Proposition~\ref{finite_order} yields $$N_G = f_A(\Nrd^{-1}_A(H'))\cap G.$$

Denote by $\widetilde{N}_G$ the group $f_A^{-1}(G)\cap\Nrd^{-1}_A(H')$. As mentioned above,~\mbox{$f_A(\widetilde{N}_G) = N_G$,} thus it remains to prove that $\widetilde{N}_G$ is finite. Note that the kernel of $\Nrd_A$ restricted to $k^*$ is finite, as well as the image of $\widetilde{N}_G$ under $\Nrd_A$. Hence the group $\widetilde{N}_G \cap k^*$ is also finite. However $\widetilde{N}_G \cap k^*$ is exactly the kernel of $f_A$ restricted to $\widetilde{N}_G$. Since the kernel and the image of $f_A$ restricted to $\widetilde{N}_G$ are finite, the group $\widetilde{N}_G$ is finite as well.
\end{proof}

We can now prove Corollary \ref{A5}.

\begin{proof}[Proof of Corollary \ref{A5}]
Let $G$ be a finite simple non-abelian subgroup of $A^*/k^*$. By Corollary~\ref{N_G} the group $N_G$ of elements of $G$ admitting a lift of finite order is a normal subgroup of $G$. Therefore, either~$G = N_G$ or $N_G$ is trivial. The latter means that $G$ is abelian, so consider the case when~\mbox{$G = N_G$.} Proposition~\ref{N_GLift} then yields that there exists a finite subgroup $\widetilde{G} \subset A^*$ such that~\mbox{$G = f_A(\widetilde{G})$.} As $G$ is simple and non-abelian, the group $\widetilde{G}$ is non-solvable. 

Denote by $\mathcal{I}^*$ the binary icosahedral group, which is a central extension of $\mathfrak{A}_5$ by $\Z/2\Z$, as in~\mbox{\cite[(5E)]{Amitsur}}. According to \cite[Cor. 4]{Amitsur} the group $\mathcal{I}^*$ is the only non-solvable group that can be a finite subgroup of a division algebra, which completes the proof.
\end{proof}

\section{p-subgroups of automorphisms of division algebras}

In this section for a division algebra $A$ we discuss the structure of finite $p$-subgroups of $A^*/k^*$ that can be lifted to finite subgroups of $A^*$. Using the results of Amitsur about the structure of finite subgroups of multiplicative groups of division algebras, we deduce that those $p$-subgroups are always cyclic for $p \geqslant 3$. We then prove a lemma that will serve us in Section~\ref{mainsect} to prove the main result.

Recall that a \textit{generalized quaternion group} $Q_{2^m}$ is the group given by generators and relations as follows: $$Q_{2^m} = \langle \, \x, \y \mid \x^{2^{m}} = 1, \; \x^{2^{m - 1}} = \y^2, \; \y \x \y^{-1} = \x^{-1}\rangle.$$ 

\begin{remark} \label{Q}
Note that its center $Z(Q_{2^m})$ is generated by $\y^2$ and the quotient $Q_{2^m}/Z(Q_{2^m})$ is isomorphic to the dihedral group $D_{2^m}$, i.\,e. the semidirect product~~\mbox{$\Z/2^{m}\Z \rtimes \Z/2\Z$,} where the action of $\Z/2\Z$ on $\Z/2^{m}\Z$ is given by $a \mapsto a^{-1}$.
\end{remark}

The proof of the following can be found in \cite[Thr. 2]{Amitsur}.

\begin{prop} \label{p-groupsAbove}
Let $A$ be a division algebra. Then all finite $p$-subgroups of $A^*$ are cyclic for every prime $p \ne 2$, and every finite $2$-subgroup of $A^*$ is either cyclic or isomorphic to a generalized quaternion group $Q_{2^m}$.
\end{prop}

Let us now prove an analogous statement for $p$-subgroups of automorphism groups of central division algebras that can be lifted to finite subgroups of division algebras.

\begin{prop} \label{p-groups}
Let $A$ be a central division algebra over a field $k$. Then every finite $p$-subgroup $N$ of $A^*/k^*$  that can be lifted to a finite group $\widetilde{N} \subset A^*$ is cyclic for $p \ne 2$. When $p = 2$, it is either cyclic or isomorphic to a dihedral group $D_{2^m}$.
\end{prop}

\begin{proof}
Note that for every prime $p$ the image of a Sylow $p$-subgroup under a surjection is a Sylow $p$-subgroup of the image. Hence, the group $N$ is the image of a $p$-subgroup of $\widetilde{N}$, thus we can suppose that $\widetilde{N}$ is itself a $p$-group. Therefore, for every $p \ne 2$, the group $N$ admits a surjection from a cyclic group $\widetilde{N}$ by Proposition~\ref{p-groupsAbove}, thus it itself is cyclic. 

It remains to consider the case when $p=2$ and $\widetilde{N}$ is a generalized quaternion group $Q_{2^m}$. Note that the kernel of the canonical projection $f_A \colon A^* \to A^*/k^*$ is the center of $A^*$, so the kernel of $f_A$ restricted to $\widetilde{N}$ is a subgroup of the center $Z(\widetilde{N})$. As mentioned in Remark \ref{Q}, the center of $Q_{2^m}$ is generated by the element of order $2$. Denote it by $x \in A^*$ and note that $x^2 - 1 = (x-1)(x+1) = 0$, thus, as $A$ is a division algebra, $x = -1$ and so is an element of $k$. Therefore $$N = f_A(\widetilde{N}) \cong Q_{2^m}/Z(Q_{2^m}) \cong D_{2^m},$$ see Remark \ref{Q}.
\end{proof}

\begin{remark} \label{conj}
Let $A$ be a central division algebra over a field $k$. Note that the action of $A^*$ on $A$ by conjugations factors through the action of $A^*/k^*$. To simplify the terminology, we will sometimes refer to the resulting action as the action of $A^*/k^*$ on $A$ by conjugations.
\end{remark}

The following lemma is crucial for the proof of Theorem~\ref{main}.

\begin{lemma} \label{subfield}
Let $A$ be a central division algebra. Let $G$ be a finite $p$-subgroup of $A^*/k^*$ for a prime $p$ and let $N_G \subset G$ be defined as in Corollary~\ref{N_G}. Then there exists a field extension $K \supset k$ such that $K \subset A$ and the following conditions hold:
\begin{itemize}
\item $K$ is invariant under the action of $G$ on $A^*$ by conjugations;
\item all the preimages of all elements of $N_G$ that act trivially on $K$, lie in $K$, and $K$ is generated by such preimages. 
\end{itemize}
\end{lemma}

\begin{proof}
Suppose first that $p \ne 2$. Then $N_G$ is cyclic by Proposition \ref{p-groups}. Let $K \supset k$ be a subalgebra of $A$ generated by $f_A^{-1}(N_G)$. Note that it is commutative as a central extension of a cyclic group. Therefore, as $A$ is a division algebra, $K$ is a subfield. Then, as $N_G$ is stable under the action of $G$ by conjugations, so is $f_A^{-1}(N_G)$ and consequently $K$. The second condition is satisfied automatically.

Now consider the case when $p = 2$. By Proposition \ref{p-groups} the group $N_G$ is either cyclic or isomorphic to the dihedral group $D_{2^m}$. In the first case the construction is the same as in the case of odd $p$. Namely, take $K$ to be the field generated by~$f_A^{-1}(N_G)$. Therefore, consider the second case. 

If $m \ne 1$, there is only one cyclic subgroup of $N_G$ of order $2^m$, so it is characteristic, that is, preserved by all automorphisms of $N_G$. Denote it by $N_G'$. In the case when $m = 1$ the group $N_G$ is isomorphic to $\Z/2\Z \times \Z/2\Z$. Consider an action of $G$ on $N_G$ by conjugations. As $G$ is a $2$-group, its image in $\Aut(N_G) \cong S_3$ is either trivial or a subgroup generated by a transposition, thus there exists a subgroup $N_G'\cong \Z/2\Z \subset N_G$ stable under conjugations by elements of $G$. In both cases $m = 1$ and $m \ne 1$ we have obtained a normal cyclic subgroup $N_G' \lhd G$ of index $2$ in $N_G$.

Let $K$ be the field generated by $f_A^{-1}(N_G') \subset A^*$. Since $N_G'$ is stable under conjugations by elements of $G$, so is~$K$. Now let us prove that no elements of $N_G$, except for the elements of $N_G'$, act trivially on~$K$. Suppose there exists an element $x \in N_G$ such that $x \not \in N_G'$ and $x$ acts trivially on~$K$. Since $N_G/N_G' \cong \Z/2\Z$, the group $N_G$ is generated by $N_G'$ and $x$, hence there exists a field extension $L \supset K$ generated by the preimage $f_A^{-1}(N_G)$. Consider a finite group $\widetilde{N}_G \subset A^*$ such that $f_A(\widetilde{N}_G) = N_G$. Since $\widetilde{N}_G \subset L^*$, we know that $\widetilde{N}_G$ is cyclic, hence so is $N_G$, which leads to a contradiction.

As no element of $N_G$, except for the elements of $N_G'$, act trivially on $K$, the second condition required in the assertion of the lemma holds.
\end{proof}

\section{Field extensions}
The goal of this section is to prove the following proposition, which, in view of the discussion in Section \ref{N_Gsect}, can be considered as a field theory analog of Theorem \ref{main}. For a given finite field extension $L \supset K$ we denote by $[L : K]$ the degree of the extension and by $\tr_{L/K}$ the trace map from $L$ to $K$. 

\begin{prop} \label{GaloisMain}
Consider a finite field extension $L \supset k$ of degree coprime with characteristic and let $f_L \colon L^* \to L^*/k^*$ be the canonical projection. Consider a finite subgroup $\mathcal{A}$ of $L^*/k^*$. Denote by $\mathcal{B}$ the subgroup of $\mathcal{A}$  consisting of all elements $x \in \mathcal{A}$ admitting a lift $\widetilde{x} \in f^{-1}(x)$ of finite order. Let $K \subset L$ be a field generated by~$f^{-1}(\mathcal{B})$. Then $[L : K]$ is divisible by $ |\mathcal{A}/\mathcal{B}|$.
\end{prop}

\begin{remark}
Note that, unlike in Section \ref{N_Gsect}, all the groups considered here are commutative, so it is straightforward that all the elements admitting a finite lift form a subgroup.
\end{remark}

We begin by proving the following lemma.

\begin{lemma} \label{GaloisLemma}
Consider a finite field extension $L \supset k$ and let $f_L \colon L^* \to L^*/k^*$ be the canonical projection. Consider a finite subgroup $\mathcal{A}$ of $L^*/k^*$ and suppose that there are no non-trivial elements $x \in \mathcal{A}$  that admit a lift $\widetilde{x} \in f_L^{-1}(x)$ of finite order. Then for every non-trivial $x \in \mathcal{A}$ and for any $\widetilde{x}\in f_A^{-1}(x)$ the trace $\tr_{L/k}(\widetilde{x}) = 0$.
\end{lemma}

\begin{proof}
Consider an element $x \in \mathcal{A}$ of order $\alpha$ and let $\widetilde{x} \in f_L^{-1}(x)$ be any of its preimages. Note that $\widetilde{x}^\alpha = a \in k$. Denote by $P(t)$ the minimal polynomial of $\widetilde{x}$. Let $m$ be its degree.

As $P(t)$ divides $t^\alpha - a$, its constant term $a_0$ is a product of $m$ roots of degree $\alpha$ of~$a$ up to a sign, so $a_0^\alpha = a^m$, hence $(\widetilde{x}^m/a_0)^\alpha = 1$ and thus $x^m$ admits a lift $\widetilde{x}^m/a_0$ of finite order. Therefore $m = \alpha$ and $P(t) = t^\alpha - a$.

Denote by $k(\widetilde{x}) \subset L$ the field subextension generated by $\widetilde{x}$ over $k$. Note that $\tr_{k(\widetilde{x})/k}(\widetilde{x}) = 0$ as the degree $1$ term coefficient of $P(t)$, therefore 
$$\tr_{L/k}(\widetilde{x}) = \tr_{k(\widetilde{x})/k}(\tr_{L/k(\widetilde{x})}(\widetilde{x})) = \tr_{k(\widetilde{x})/k}([L: k(\widetilde{x})]\cdot \widetilde{x}) = 0,$$

which concludes the proof.
\end{proof}

Let us now use Lemma \ref{GaloisLemma} to prove the special case of Proposition \ref{GaloisMain} when the group $\mathcal{B}$ is trivial.

\begin{prop} \label{GaloisAdd}
Consider a finite field extension $L \supset k$ of degree coprime with characteristic and let $f_L \colon L^* \to L^*/k^*$ be the canonical projection. Consider a finite subgroup $\mathcal{A}$ of $L^*/k^*$ and suppose that there are no non-trivial elements $x \in \mathcal{A}$  that admit a lift $\widetilde{x} \in f_L^{-1}(x)$ of finite order. Then $[L : k]$ is divisible by $|\mathcal{A}|$.
\end{prop}

\begin{proof}

For every element $x_i \in \mathcal{A}$ take an arbitrary lift $\widetilde{x}_i \in f_L^{-1}(x_i)$. Let us prove that all the $\widetilde{x}_i$ are linearly independent over $k$. 

Consider the converse. Then there exists an equality of the form $a_0\widetilde{x}_0 + a_1\widetilde{x}_1 + \ldots + a_l\widetilde{x}_l = 0$ where $\widetilde{x}_i \in L$ and $a_i \in k$ are non-zero. Dividing by $\widetilde{x}_0$ we obtain 
\begin{equation} \label{eq1}
a_0 + a_1 \frac{\widetilde{x}_1}{\widetilde{x}_0} + \ldots + a_l\frac{\widetilde{x}_l}{\widetilde{x}_0} = 0.
\end{equation}
As $\frac{\widetilde{x}_i}{\widetilde{x}_0}$ is a preimage of $x_ix_0^{-1} \in \mathcal{A}$, Lemma \ref{GaloisLemma} yields $$\tr_{L/k}\left(\frac{\widetilde{x}_i}{\widetilde{x}_0}\right ) = 0.$$ 
Thus, calculating the trace of both sides of \eqref{eq1}, we obtain $\tr_{L/k}(a_0) = 0$, which leads to a contradiction as the degree of the field extension is coprime with characteristic. Therefore the elements $\widetilde{x}_i$ are linearly independent over $k$. 

Consider a vector space $V \subset L$ of dimension $|\mathcal{A}|$ generated by the elements $\widetilde{x}_i$ over $k$. Note that~$V = f_L^{-1}(\mathcal{A}) \cup 0$, thus it is an algebra and contains the inverses of all the elements, therefore it is a field. We thus conclude by $$[L : k] = [V : k][L : V] = |\mathcal{A}|[L : V].$$\\[-1 cm]
\end{proof}

We are now ready to prove the principal proposition of the section.

\begin{proof}[Proof of Proposition \ref{GaloisMain}]
Denote by $f \colon L^* \to L^*/K^*$ the canonical projection and let $\mathcal{C}$ be the image $f(\mathcal{A})$. Note that as $\mathcal{B} \subset K^*$ the map $f|_\mathcal{A}$ factors through $\mathcal{A}/\mathcal{B}$, therefore $|\mathcal{C}|$ divides $|\mathcal{A}/\mathcal{B}|$. We prove that the group $\mathcal{C} \subset L^*/K^*$ satisfies the conditions of Proposition \ref{GaloisAdd}. Indeed, consider $x \in \mathcal{C}$ and suppose $\widetilde{x} \in f^{-1}(x)$ is its preimage of finite order. However, $\mathcal{B}$ contains all the elements of $\mathcal{A}$ that can be lifted to elements of finite order in $L$, therefore $f_L(\widetilde{x})$ is an element of $\mathcal{B}$, and thus $f(\widetilde{x}) = 1$ in~$\mathcal{C}$. We then apply Proposition \ref{GaloisAdd} to conclude the proof.
\end{proof}

\section{The case of trivial N$_G$}

This section is devoted to the proof of the following proposition, which is a special case of Theorem~\ref{main}.

\begin{prop} \label{CV}
Let $A$ be a central division algebra of degree $n$ over a field $k$ of characteristic coprime with $n$. Let $G$ be a finite subgroup of $A^*/k^*$ such that the group $N_G$ defined in Corollary~\ref{N_G} is trivial. Then $|G|$ divides $n^2$.
\end{prop}

A very similar statement is proved by Shramov and Vologodsky in \cite[Cor. 4.4]{CostyaVologodskiy}. We provide their proof, modified slightly, in order to make it fit our context and to remove unnecessary restrictions on the base field.

Let us begin with a construction.

\begin{construction} \label{form}
Let $A$ be a central division algebra of degree $n$ over a field $k$ and let $\mathcal{A}$ be a finite abelian subgroup of $A^*/k^*$. Define a bilinear form $\beta \colon \mathcal{A} \times \mathcal{A} \to k$ as follows: for any two elements $x$, $y \in \mathcal{A}$ take $$\beta(x, y) = \widetilde{x}\widetilde{y}\widetilde{x}^{-1}\widetilde{y}^{-1}$$ for any $\widetilde{x} \in f_A^{-1}(x)$ and $\widetilde{y} \in f_A^{-1}(y)$. Note that, as the kernel of $f_A$ is the center of $A^*$, the form~$\beta$ does not depend on the choice of preimages. Moreover, as $\mathcal{A}$ is abelian, $$f_A(\widetilde{x}\widetilde{y}\widetilde{x}^{-1}\widetilde{y}^{-1}) = xyx^{-1}y^{-1} = 1,$$ thus the image of $\beta$ is a subset of $\ker f_A = k^*$. 
\end{construction}

\begin{remark}
To check that the form $\beta$ defined in Construction \ref{form} is bilinear, consider $x, y, z \in \mathcal{A}$ and take any of their preimages $\widetilde{x}, \widetilde{y}, \widetilde{z} \in A^*$. Since $\widetilde{x}\widetilde{z} = \widetilde{z}\widetilde{x}\beta(x, z)$ and $\widetilde{y}\widetilde{z} = \widetilde{z}\widetilde{y}\beta(y, z)$, we conclude by $$\widetilde{z}\widetilde{x}\widetilde{y}\beta(xy, z) = \widetilde{x}\widetilde{y}\widetilde{z} = \widetilde{x}\widetilde{z}\widetilde{y}\beta(y, z) = \widetilde{z}\widetilde{x}\widetilde{y}\beta(y, z)\beta(x, z).$$
\end{remark}

The proof of the following lemma is similar to \cite[Lem. 4.3]{CostyaVologodskiy}.

\begin{lemma} \label{Gamma}
Let $A$ be a central division algebra over a field $k$ and let $\mathcal{A}$ be a finite abelian subgroup of $A^*/k^*$. Denote by $\beta \colon \mathcal{A} \times \mathcal{A} \to k$ the form defined in Construction \ref{form}. Then there exists a subgroup $\Gamma \subset \mathcal{A}$ such that
\begin{enumerate}
\item[\rm(i)] $\beta(x, y) = 1$ for every $x, y \in \Gamma$; \label{form1}
\item[\rm(ii)] $|\Gamma|^2$ is divisible by $|\mathcal{A}|$; 
\item[\rm(iii)] the subalgebra generated by $f_A^{-1}(\Gamma)$ is a subfield of $A$. \label{form2}\label{form3}
\end{enumerate}
\end{lemma}

\begin{proof}
We proceed by induction on the order of $\mathcal{A}$. Consider an element $x \in \mathcal{A}$ of the largest order~$m$ and denote by $\beta_x \colon \mathcal{A} \to k^*$ the group morphism that takes $y \in \mathcal{A}$ to $\beta(x, y)$. Denote by~$\mathcal{A}_x$ its kernel and note that, since $\beta^m(x, y) = \beta(x^m, y) = 1$, the image of $\beta_x$ consists of elements of order dividing $m$ in~$k^*$, hence the inclusion $\mathcal{A}/\mathcal{A}_x \hookrightarrow \Z/m\Z$. Moreover, as $\beta(x, x) = 1$, there is an inclusion $\langle x \rangle \hookrightarrow \mathcal{A}_x$. 

Note that in a finite abelian group the subgroup generated by an element of maximal order is always a direct summand, so we can present the group $\mathcal{A}_x$ in the form $\mathcal{A}_x = \langle x \rangle \times \mathcal{A}'$. By induction we can assume that the group $\mathcal{A}'$ contains a subgroup $\Gamma'$ satisfying the conditions above. 

We denote by $\Gamma$ the product $\Gamma' \times \langle x \rangle$ and prove the assertion of the lemma. One can easily check that $\beta(y, z) = 1$ for every $y$, $z \in \Gamma$. Let us show that $|\mathcal{A}|$ divides $|\Gamma|^2$. Indeed, using the inclusion $\mathcal{A}/\mathcal{A}_x \hookrightarrow \Z/m\Z$ one deduces that $|\mathcal{A}|$ divides $m|\mathcal{A}_x| = m^2|\mathcal{A}'|$. By induction, $|\mathcal{A}'|$ divides $|\Gamma'|^2$, thus $|\mathcal{A}|$ divides $m^2|\Gamma'|^2 = |\Gamma|^2$. 

Finally, by induction, the subalgebra generated by $f_A^{-1}(\Gamma')$ is a field. Denote it by $K$. Note that as

 $$\widetilde{x}\widetilde{y}\widetilde{x}^{-1}\widetilde{y}^{-1} = \beta(f_A(\widetilde{x}), f_A(\widetilde{y})) = 1$$ 
 for any preimage $\widetilde{x}$ of $x$ and every $\widetilde{y} \in f_A^{-1}(\Gamma')$, every element in the preimage of $\Gamma'$ commutes with $\widetilde{x}$. 
 Therefore, the subalgebra generated by $f_A^{-1}(\Gamma)$ is of the form $K[t]/P(t)$, where $P(t)$ is the minimal polynomial of $\widetilde{x}$ over $K$. As $A$ is a division algebra, $P(t)$ is irreducible over $K$, therefore the subalgebra generated by $f_A^{-1}(\Gamma)$ is a field.
\end{proof}

Now we are ready to prove Proposition \ref{CV}.

\begin{proof}[Proof of Proposition \ref{CV}]
Note that $G$ is abelian by Corollary \ref{N_G}. Let $\Gamma \subset G$ be the group constructed in Lemma \ref{Gamma} and $L \supset k$ be the subfield of $A$ generated by $f_A^{-1}(\Gamma)$. Since $N_G$ is trivial, no element of $G$ can be lifted to an element of $A^*$ of finite order, so we can apply Proposition~\ref{GaloisAdd} to $\Gamma$ and $L \supset k$, thus obtaining that $|\Gamma|$ divides $[L : k]$. Recall that the degree of every subfield of a central simple algebra $A$ divides the degree of $A$, hence $|\Gamma|$ divides $n$. Lemma~\ref{Gamma}\,(ii) yields that $|\Gamma|^2$ is divisible by $|G|$, and therefore $n^2$ is divisible by $|G|$. 
\end{proof}

\section{General case} \label{mainsect}
In this section we prove the generalization of Proposition \ref{CV}, thus completing the proof of Theorem~\ref{main}.

\begin{theorem} \label{n^2}
Let $A$ be a central division algebra of degree $n$ over a field $k$ of characteristic coprime with $n$. Let $G$ be a finite subgroup of $A^*/k^*$. Consider the subgroup $N_G$ of $G$ constructed in Corollary \ref{N_G}. Then $|G/N_G|$ divides~$n^2$.
\end{theorem}

Before proving the theorem we consider the case when $G$ is a $p$-group.

\begin{theorem} \label{p-n^2}
Let $A$ be a central division algebra of degree $n$ over a field $k$ of characteristic coprime with $n$ and let $G \subset A^*/k^*$ be a finite $p$-group for a prime $p$. Then $|G/N_G|$ divides $n^2$.
\end{theorem}

\begin{proof}
Consider the field $K \subset A$ given by Lemma \ref{subfield}. Denote by $\varphi \colon G \to \Aut(K/k)$ the map corresponding to the action of $G$ on $K$ by conjugations. The degree $d$ of $K$ over $k$ is divisible by~$|\Aut(K/k)|$ and thus is divisible by $|\varphi(G)|$.

We now want to study the kernel $\ker \varphi$. Let $B$ be a subalgebra of $A$ generated by $f_A^{-1}(\ker \varphi)$ and let $L$ be the center of $B$. Note that every element in $f_A^{-1}(\ker \varphi)$ commutes with $K$ by the definition of $\varphi$, so $K$ is a subfield of $L$. Denote by $q$ the degree of the extension $L \supset k$. Let~$\psi \colon B^*/k^* \to B^*/L^*$ be the canonical projection. Then the canonical projection \mbox{$f_B \colon B^* \to B^*/L^*$} can be seen as the composition~$\psi \circ f_A|_B$.

For convenience consider the following commutative diagram of groups. \\[-0.6cm]

$$\xymatrix@C-15pt{
f_A^{-1}(\ker\varphi)\ar@/_1.3pc/_(0.28){f_B}[dd] \ar[r] \ar[d]^{f_A} & B^* \ar@/_1.3pc/_(0.28){f_B}[dd] \ar[r] \ar[d]^{f_A} & A^* \ar[d]^{f_A} \\
\ker \varphi \ar[r] \ar[d]^{\psi} & B^*/k^* \ar[r] \ar[d]^{\psi} & A^*/k^* \\
H \ar[r] & B^*/L^* \\
}$$

Let $N_G' \subset N_G$ be the maximal subgroup of $N_G$ such that $f_A^{-1}(N_G') \subset K$, which is, by Lemma~\ref{subfield}, the group $\ker \varphi \cap N_G$. Consider the group $$\mathcal{A} = G \cap f_A(L^*).$$ Note that $\mathcal{A} \subset f_A(L^*) = L^*/k^*$ so we can apply Proposition \ref{GaloisMain} to $\mathcal{A}$ and the extension $L \supset k$. Recall that by definition of $N_G$ an element $x \in \mathcal{A}$ has a preimage $\widetilde{x} \in f_A^{-1}(x)$ of finite order if and only if $x \in N_G$, so the group $\mathcal{B}$ from Proposition \ref{GaloisMain} is exactly $$\mathcal{A} \cap N_G = f_A(L^*) \cap N_G= N_G'$$ and the field generated by its preimage is $K$. Hence, Proposition~\ref{GaloisMain} yields that $$|\ker (\psi|_{\ker\varphi})/N_G'| = |(\ker \varphi \cap f_A(L^*))/N_G'|$$ divides $[L : K] = \frac{q}{d}.$ \\[-0.3cm]

Consider the central division algebra $B$ over $L$ and denote by $H$ the finite group $$\psi(\ker\varphi) \subset B^*/L^*.$$ In this setting consider the group $N_H \lhd H$ constructed in Section~\ref{N_Gsect}. Note that an element $x \in H$ has a preimage $\widetilde{x}\in B$ of finite order if and only if it has a preimage $y \in  \psi^{-1}(x)$ that lies in $N_G$. Therefore, $$N_H = \psi(\ker\varphi \cap N_G) = \psi(N_G') = \{1\},$$ as $N_G' \subset f_A(L^*)$. Proposition~\ref{CV} applied to $H$ thus yields that $|H|$ divides $$\deg_L^2B = \dim_LB = \frac{\dim_kB}{\dim_kL}.$$ Recall that the dimension of a simple subalgebra $B$ of a central simple algebra $A$ divides the dimension of $A$, and therefore $|H|$ divides $$\frac{\dim_kA}{\dim_kL} = \frac{\dim_kA}{q} = \frac{n^2}{q}.$$

To summarize,
$$|\ker\varphi/N_G'| = |\ker (\psi|_{\ker\varphi})/N_G'| \cdot |\im(\psi|_{\ker\varphi})|$$ divides $$\frac{q}{d} \cdot \frac{n^2}{q} = \frac{n^2}{d}.$$
Moreover, $|\im\varphi|$ divides $d$, so $$|G/N_G'| = |\im \varphi|\cdot|\ker\varphi/N_G'|$$ divides~$n^2$. Finally, since $N_G'$ is a subgroup of $N_G$, we obtain that $|G/N_G|$ divides $|G/N_G'|$, hence it divides $n^2$.
\end{proof}

Now we are ready to prove Theorem \ref{n^2} for an arbitrary group $G$.

\begin{proof}[Proof of Theorem \ref{n^2}]
For a group $H$ and a prime $p$ denote by $|H|_p$ the number $p^\alpha$ such that $|H|$ is divisible by $p^\alpha$ and is not divisible by $p^{\alpha + 1}$.

For a prime $p$ consider a Sylow $p$-subgroup $G_p \subset G$. Recall that $N_G$ is the set of all elements $x \in G$ such that there exists $\widetilde{x}\in f_A^{-1}(x)$ of finite order, so $N_{G_p} = N_G \cap G_p$. 

Therefore, as the image of a Sylow $p$-subgroup under a surjection is a Sylow $p$-subgroup,
we conclude that $$|G/N_G|_p = |G_p/(G_p \cap N_G)| = |G_p/N_{G_p}|,$$ hence Theorem \ref{p-n^2} yields that $|G/N_G|_p$ divides $n^2$ for all prime $p$, and thus $|G/N_G|$ divides $n^2$.
\end{proof}

Theorem~\ref{main} directly follows from Theorem~\ref{n^2} combined with the results of Section~\ref{N_Gsect}.

\section{Division algebras of prime degree}

In this section we apply Theorem~\ref{main} to prove a generalization of the result of Shramov (see~\cite{surfaces}), describing all finite groups acting on non-trivial (i.e.\;those not isomorphic to matrix algebras over the base field) central simple algebras of prime degree. Note that by Wedderburn theorem every non-trivial central simple algebra of prime degree is a division algebra, so we can apply our previous results.

We start by reminding the reader a theorem from a general theory of division algebras.

\begin{theorem}[{see~\cite[§VIII.10.3]{bou}}]\label{bou}
Let $A$ be a division algebra of degree $n$ over a base field~$k$. Let $K \subset A$ be a maximal field, containing $k$, in $A$. Then the degree $[K \colon k]$ is equal to $n$.
\end{theorem}

Let us begin by considering $p$-subgroups of finite subgroups of automorphisms of central division algebras of prime degree.

\begin{prop}\label{c1}
Let $A$ be a non-trivial central simple algebra of prime degree $p \geqslant 3$ over a field~$k$. For a prime $q$, let $N \subset A^*/k^*$ be a finite $q$-subgroup that admits a finite lift $\widetilde{N} \subset A^*$. Then $N$ is cyclic.
\end{prop}

\begin{proof}
By Proposition~\ref{p-groups} it is enough to prove that $N$ cannot be isomorphic to $D_{2^m}$. 

If $N$ contains an element of order $2$, there exists an element $x \in A$ such that $x \not \in k$ and~$x^2 \in k$, and thus the field extension $k(x) \supset k$ has degree $2$. In this case Theorem \ref{bou} yields that $2$ should divide $p$, which leads to a contradiction.
\end{proof}

Let us now state the following result from group theory. For a proof see \cite[p.145, Thr. 11]{groups}.

\begin{theorem} \label{groups}
Let $N$ be a finite group that has no non-cyclic Sylow $p$-subgroups for any prime~$p$. Then for some coprime $m$, $n \in \mathbb{N}$ there exists a short exact sequence
$$1 \to \Z/n\Z \to N \to \Z/m\Z \to 1.$$
\end{theorem}

\begin{corollary} \label{char}
Let $N$ be a non-trivial finite group such that it has no non-cyclic Sylow $p$-subgroups for any prime $p$. Then $N$ contains a non-trivial cyclic characteristic subgroup.
\end{corollary}

\begin{proof}
Using Theorem \ref{groups} we obtain a short exact sequence 
$$1 \to H \to N \to \Z/m\Z \to 1,$$
where $H \cong \Z/n\Z$ such that $n$ and $m$ are coprime. In case when $n = 1$ the group $N$ is cyclic and the proposition follows immediately, so let us consider the case when $n \ne 1$.

If an element $x \in N$ has a non-trivial image under the surjection to $\Z/m\Z$, then its order is not coprime with $m$. Therefore, the only subgroup of order $n$ in $N$ is the image of $H$, so it is characteristic.
\end{proof}

\begin{prop} \label{normal}
Let $A$ be a non-trivial division algebra of prime degree $p \geqslant 3$ over a field~$k$. Then every non-trivial finite subgroup $G \subset A^*/k^*$ contains a non-trivial normal cyclic subgroup.
\end{prop}

\begin{proof}
Let $N_G \lhd G$ be the subgroup of $G$ defined in Section \ref{N_Gsect}. In the case when $N_G$ is trivial, the group $G$ is abelian by Lemma \ref{N_G}, so the claim follows automatically. Thus consider the case of non-trivial $N_G$. In this case, Proposition \ref{c1} and Corollary \ref{char} yield that $N_G$ contains a non-trivial cyclic characteristic subgroup, which will therefore be a normal subgroup of $G$.
\end{proof}


\begin{prop} \label{fields}
Let $K \supset k$ be a field extension of prime degree $p$ and let $\mathcal{A} \subset K^*/k^*$ be a finite group. Then the group $\mathcal{A}$ can be presented in one of the forms $\Z/n\Z \times \Z/p\Z \cong \Z/np\Z$ or $\Z/n\Z$, where $n$ is coprime with $p$. Moreover, an element $x \in \mathcal{A}$ can be lifted to an element of $K$ of finite order if and only if it lies in~$\Z/n\Z$.
\end{prop}

\begin{proof}
Let $\mathcal{B} \subset \mathcal{A}$ be the subgroup of all elements $x$ such that there exists an element $\widetilde{x} \in K^*$ of finite order mapped to $x$ by the canonical projection as in Proposition~\ref{GaloisMain}. Generating a subgroup of $K^*$ by such preimages $\widetilde{x}$ of every $x \in \mathcal{B}$ provides us with a finite subgroup $\widetilde{\mathcal{B}} \subset K^*$ that surjects onto $\mathcal{B}$ by the canonical projection. As every finite subgroup of the multiplicative group of a field is cyclic, the group $\widetilde{\mathcal{B}}$ is cyclic and thus $\mathcal{B}$ is cyclic as well. Denote by $n$ its order. Note that, moreover, there is no element $x$ of order $p$ in $K^*$ that does not lie in $k^*$, as otherwise it would generate a subextension $k \subset K' \subset K$ of smaller degree. Therefore, the order of $\widetilde{\mathcal{B}}$ is coprime with~$p$, and thus so is $n$.

We now apply Proposition \ref{GaloisMain} to obtain that the quotient $\mathcal{A}/\mathcal{B}$ is a subgroup of $\Z/p\Z$, which, as $\mathcal{A}$ is abelian, completes the proof of the assertion.
\end{proof}

As before, we start by considering the case when the finite group $G\subset A^*/k^*$ contains no non-trivial elements that admit a lift of finite order.

\begin{lemma} \label{triv}
Let $A$ be a non-trivial division algebra of prime degree $p \geqslant 3$ over a field $k$ of characteristic coprime with $p$. Consider its finite subgroup $G \subset A^*$. If its normal subgroup $N_G$ defined in Corollary \ref{N_G} is trivial, then~\mbox{$G \subset \Z/p\Z \times \Z/p\Z$.}
\end{lemma}

\begin{proof}
Let $x \in G$ be a non-trivial element and let $K \supset k$ be the field extension generated by~$f_A^{-1}(x)$. Theorem~\ref{bou} then yields that the degree~$[K \colon k] = p$. As $N_G$ is trivial, no non-trivial element of $G$ admits a preimage of finite order, thus we can apply Proposition~\ref{GaloisAdd} to the group generated by $x$ and the field $K$ to obtain that the order of $x$ is exactly~$p$. By Proposition \ref{CV} the order of $G$ divides $p^2$ therefore, as every element of $G$ has order at most $p$, the group $G$ is a subgroup of $\Z/p\Z \times \Z/p\Z$.
\end{proof}

We are now ready to prove the principal theorem of the section.

\begin{proof}[Proof of Theorem \ref{mainPrime}]
Consider the subgroup $N_G \lhd G$ defined in Corollary \ref{N_G}. In the case when $N_G$ is trivial Lemma \ref{triv} yields the claim.

Suppose now that $N_G$ is non-trivial. Consider a non-trivial normal cyclic subgroup $\Theta \lhd G$ obtained by Proposition \ref{normal}. Consider the subfield $K \subset A$ generated over $k$ by the preimage of $\Theta$ under the canonical projection. Since $\Theta$ is a normal subgroup of $G$, the group $G$ acts on $K$ by conjugations, which yields a map $\varphi \colon G \to \Aut(K/k)$. Denote by $G'$ its kernel. By Theorem \ref{bou} every non-trivial field extension of $k$ contained in $A$ is maximal, hence any lift of any element of $G'$ lies in $K^*$, and therefore $G' \subset K^*/k^*$. We now apply Proposition \ref{fields} to get that $G'$ is isomorphic either to $\Z/n\Z$ or to $\Z/n\Z \times \Z/p\Z$, where $n$ is coprime with $p$ and the unique subgroup of $G'$ isomorphic to $\Z/n\Z$ consists of all the elements of $G'$ that admit a lift of finite order in $K^*$.

Note that the group $\Aut(K/k)$ is either $\Z/p\Z$ or trivial. Consider the case when $n = 1$. The group $G$ then has the order dividing $p^2$, so to prove the assertion it remains to exclude the $\Z/p^2\Z$ case. Suppose that $G \cong \Z/p^2\Z$. As $n = 1$, no element of $G' \cong \Z/p\Z \subset \Z/p^2\Z$ admits a lift of finite order, therefore the group $N_G$ is trivial, which contradicts the assumption. This completes the proof in the case when~$n = 1$. 

Suppose now that $n \ne 1$. In this case the extension $K \supset k$ is cyclotomic, as it can be generated by the lift of the generator of $\Z/n\Z$ of finite order, thus it is Galois. Recall from Theorem~\ref{bou} that $K$ is of degree $p$ over~$k$, hence its group of automorphisms $\Gal(K/k)$ is isomorphic to~$\Z/p\Z$. Denote by $\widetilde{x} \in K$ any preimage of finite order of the generator $x$ of $\Z/n\Z$. Note that $\widetilde{x}$ is a root of unity. The Galois group $\Gal(K/k)$ then acts on the group generated by $\widetilde{x}$ in such a way that the only elements preserved by this action are those contained in $k^*$. Taking quotient by $k^*$ we thus obtain that the action of $\Gal(K/k) \cong \Z/p\Z$ on the group generated by $\widetilde{x}$ induces an action on $\Z/n\Z$ such that the resulting semidirect product $$\Z/n\Z \rtimes \Gal(K/k) \cong \Z/n\Z \rtimes \Z/p\Z$$ is balanced, as every element of $\Z/p\Z$ acts non-trivially on all the elements of $\Z/n\Z$.

We can now conclude the proof considering the image of $G$ under $\varphi$ in $\Gal(K/k) \cong \Z/p\Z$. If it is trivial the proof is complete, so consider the case when $\im \varphi \cong \Z/p\Z$. Then $G$ fits into a short exact sequence
$$1 \to G' \to G \xrightarrow{\varphi} \Z/p\Z \to 1.$$
We proceed by proving that $\varphi$ admits a right inverse. Consider any preimage $t \in G$ under $\varphi$ of the generator of the group $\Gal(K/k)$. By the discussion above, $t$ does not commute with any non-trivial element of $\Z/n\Z$, thus the order of $t$ is either $p$ or $p^2$. Due to Proposition~\ref{fields} no element of order $p$ can admit a lift of finite order in $A$, so the order of $t$ cannot be equal to $p^2$ by Proposition~\ref{GaloisAdd} applied to the subfield generated by $f_A^{-1}(t)$. Therefore the order of $t$ is $p$ and there exists an inclusion $\im \varphi \hookrightarrow G$, and thus $G$ is isomorphic either to $\Z/n\Z \rtimes \Z/p\Z$ or to
$$(\Z/n\Z \times \Z/p\Z) \rtimes \Z/p\Z \cong (\Z/n\Z \rtimes \Z/p\Z) \times \Z/p\Z.$$ 
The latter isomorphism is due to the fact that there is no non-trivial action of $\Z/p\Z$ on itself by automorphisms. Moreover, as discussed above, the semidirect products in both of these cases are balanced.
\end{proof}

The constructions of examples of possible finite subgroups acting on division algebras of degree~$3$ are presented in \cite[Sect. 4]{surfaces}. One can generalize them for division algebras of any prime degree to obtain that all the groups mentioned in Theorem~\ref{mainPrime} are subgroups of automorphism groups of division algebras over specific fields, at least in characteristic~$0$. These examples furthermore provide illustration to Theorem~\ref{main}: for the group $(\Z/n\Z \rtimes \Z/p\Z) \times \Z/p\Z$, constructed this way, in the notation of Theorem~\ref{main} the group $N_G$ is isomorphic to $\Z/n\Z$.

\begin{remark}
Note that we have, moreover, obtained that, in the setting of Theorem~\ref{mainPrime}, the group $\Z/n\Z$ with $n$ coprime with $p$ acts on $A$ if and only if $A$ contains a field extension $K$ over $k$ containing the $n$-th root of unity $\xi$, such that $k$ does not contain any of its non-trivial powers.
\end{remark}

\end{document}